\documentclass[11pt,twoside]{article}

\def\titlename{\huge Approximate amenability  of Segal algebras II}

\title{\titlename}

\def\authname{Mahmood Alaghmandan}
\author{{\normalsize\sc \authname}}

\usepackage{amsmath}

\usepackage{yfonts}

\usepackage{xcolor}
\definecolor{blue1}{RGB}{32,78,170}
\definecolor{blue2}{RGB}{93,92,160}
\definecolor{blue3}{RGB}{40,51,202}
\definecolor{blue4}{RGB}{0,0,0}
\definecolor{purple1}{RGB}{128,0,128}

\usepackage[all]{xy}

\definecolor{El}{rgb}{.4,.9,1}
\usepackage[pdftex,
            pdfauthor={Mamhood},
            linkcolor=blue1,
            citecolor=blue3,
            colorlinks=true]{hyperref}

\usepackage{titlesec}

\titleformat{\section}
  {\normalfont\fontsize{12}{15}\bfseries}{\thesection}{1em}{}

\usepackage[english]{babel}
\usepackage{blindtext}
\usepackage[scaled=.90]{helvet}% Helvetica, served as a model for arial
\usepackage{xcolor}
\usepackage{titlesec}
\titleformat{\chapter}[display]
  {\normalfont\sffamily\huge\bfseries\color{blue4}}
  {\chaptertitlename\ \thechapter}{20pt}{\Huge}
\titleformat{\section}
  {\normalfont\sffamily\Large\color{blue4}}
{\thesection}{1em}{}

 \titleformat{\subsection}
  {\normalfont\sffamily\Large\color{blue4}}
  {\thesubsection}{1em}{}

\newcommand{\ma}[1]{\textcolor{blue2}{\textsf{#1}}}

            \newcounter{pulse}[section]
\newtheorem{theorem}[pulse]{\bf \textsf{Theorem}}

\newtheorem{lemma}[pulse]{\bf \textsf{Lemma}}

\newtheorem{prop}[pulse]{\bf \textsf{Proposition}}

\newtheorem{dummy-eg}[pulse]{\bf \textsf{Example}}
\newtheorem{dummy-rem}[pulse]{\bf \textsf{Remark}}

\newtheorem{dummy-def}[pulse]{\bf \textsf{Definition}}

\newcommand{\ques}{{\medskip\par\noindent {\bf \textsf{Question.}~}}}
\usepackage{amsmath,amssymb}

\newenvironment{proof}{\noindent{\it Proof.}\/}{\hfill$\Box$\newline\ignorespacesafterend}

\newcommand{\supp}{\operatorname{supp}}

\newcommand{\cA}{{\cal A}}

\newcommand{\Sp}[2]{\Delta({#1} {#2})}
\newcommand{\norm}[1]{\Vert #1 \Vert}
\begin{document}

\maketitle
 \begin{abstract}
We prove that every proper Segal algebra of a SIN group  is not approximately amenable. 

 \end{abstract}
 
%\tableofcontents
%\vfill\eject

\vskip2.5em

Gourdeau in \cite{go2} showed that a Banach algebra $A$ is \ma{amenable} if and only if every bounded derivation $D:A\rightarrow X$ for any Banach $A$-bimodule $X$ can be approximated by a net of inner derivations. A weaker version of  this notion is \ma{approximate amenability} of Banach algebras that is, a Banach algebra $A$ is  {approximately amenable} if and only if every bounded derivation $D:A\rightarrow X^*$ for every Banach dual $A$-bimodule $X^*$ can be approximated by a net of inner derivations.  The concept of approximate amenability first was introduced and studied in \cite{aa}. 

Similar to amenability, different algebras were studied for their approximate amenability property including \ma{Segal algebras} (for the definition of Segal algebras and their basic properties, look at \cite{re2}.) In \cite{da}, Dales and Loy studied some specific Segal algebras on the commutative groups $\Bbb{T}$ and $\Bbb{R}$. They   proved that those Segal algebras are not approximately amenable; consequently,  they suggested that the same should be true for every \ma{proper} Segal algebra on these groups.  We call a Segal algebra proper if it not equal to the group algebra.

Subsequently, Choi and Ghahramani, \cite{ch}, proved that this conjecture is true by showing that   proper Segal algebras on $\Bbb{T}^d$ and $\Bbb{R}^d$ are not approximately amenable (for any dimension $d$). To do so, they developed a criterion  for ``ruling out approximate amenability" of Banach algebras. 
Ghahramani, in the  Banach algebra 2011 conference,  conjectured that every proper Segal algebra of a locally compact group cannot be approximately amenable.

In \cite{ma}, the author applied the criterion developed in \cite{ch} to show that in fact, every proper Segal algebra of a locally compact abelian group is not approximately amenable. Also, applying the hypergroup structure on the dual of compact groups, it was proved that for some classes of compact groups, including $\operatorname{SU(2)}$, every proper Segal algebra is not approximately amenable. 

In this short manuscript we prove that this conjecture is actually true for every SIN group. Recall that a locally compact group $G$ is called a \ma{SIN group} if there exists a topological basis of conjugate invariant neighbourhoods of the identity element of the group $G$. This class of locally compact groups includes abelian, compact, and discrete groups.

\begin{theorem}\label{t:main-thm}
Every proper Segal algebra of a SIN group is not approximately amenable.
\end{theorem}

We prove this theorem, for a generalized version of Segal algebras, called \ma{abstract Segal algebras}.  A Banach algebra $({B},\|\cdot\|_{{B}})$ is  an { abstract Segal algebra} of a Banach algebra 
$({A},\|\cdot\|_{A})$ if 
 ${B}$ is a dense left ideal in ${\cA}$,
 there exists $C>0$ such that $\|b\|_{\cA} \leq C
\|b\|_{B}$
(for each $b\in {B}$), and 
there exists $M>0$ such that $\|ab\|_{B}\leq
M\|a\|_{\cA}\|b\|_{B}$ for all  $a,b \in {B}$. We call $B$ a \ma{proper} abstract  Segal algebra of $A$ if $B\neq A$.
It is clear that every (proper) Segal algebra of a locally compact group $G$ is a (proper) abstract Segal algebra of $L^1(G)$. 

\vskip.5em

Let $A$ be a commutative Banach algebra and let $\Sp{ }{A}$ denote the   \ma{Gelfand spectrum} of $A$ and for each $a\in A$, $\widehat{a}$ is the Gelfand transform of $a$.  For definitions related to the Gelfand spectrum of commutative Banach algebras, we refer to \cite{ka}.
%A commutative Banach algebra $A$ is called \ma{semisimple} if $\cap_{x\in \Sp{}{A}} \ker x = \emptyset$.  
%Let $A$ be a commutative Banach algebra and $I_0:=\cap_{x\in \Sp{}{A}} \ker x$. Then $I_0$ is a closed ideal in $A$ and one may consider the new Banach algebra $A/I_0$ equipped with quotient  norm that is $\norm{A/I_0}{a+I_0}:=\inf_{b\in I_0} \norm{A}{a+b}$. Then $A/I_0$ forms a commutative semisimple Banach algebra whose maximal ideal space is  isomorphic to $\Sp{}{A}$. Therefore, from now on,  we assume that all commutative Banach algebras are semisimple. 
%Using the Gelfand representation theory, for a semisimple regular commutative  Banach algebra $\cA$, $\widehat{\cA}:=\{\widehat{a}:\; a\in\cA\}$ forms an algebra of continuous functions on its maximal ideal space vanishing at infinity, $\widehat{\cA}\subseteq C_0(\oA)$.  
%One may notice that if $\cA$ is a semisimple regular commutative Banach algebra, $\widehat{\cA}$ is dense in $C_0(\oA)$. Let us define for $a\in \cA$, $\coz(\widehat{a}):=\{x\in\oA:\ \widehat{a}(x)\neq 0\}$.
We denote the set of all elements  $a\in A$ such that $\supp(\widehat{a})$  is compact by $A_c$. A semisimple commutative Banach algebra $A$ is called a \ma{Tauberian algebra} when  $A_c$ is dense in $A$.
%\begin{dfn}\label{d:regularity}
%A commutative Banach algebra $A$ is called \ma{regular}, if for every $x\in\Sp{}{A}$  and every open neighbourhood $U$ of $x$ in the Gelfand topology, there exists an element $a\in {A}$ such that $\widehat{a}(x)=1$ and $\widehat{a}$ is zero on $ \Sp{}{A}\setminus U$. Let $A$ be a semisimple regular commutative  Banach algebra and let $B$ be an abstract Segal algebra of $ A$. Then $B$ is also a semisimple regular algebra whose maximal ideal space is $\Sp{}{A}$. Moreover, $A_c\subseteq B$, (see \cite[Proposition~2.2]{ma}). 
%\end{dfn}

 For a  Banach algebra $A$ and a constant $D> 0$, $A$ has a \ma{$D$-bounded approximate identity} if there is a net $(e_\alpha)_{\alpha}\subseteq A$ such that for every $a\in A$, $\norm{ae_\alpha - a}_A \rightarrow 0$, $\norm{e_\alpha a - a}_A \rightarrow 0$, and $\sup_\alpha \norm{e_\alpha}_A\leq D$.
Note that if $A$ is a unital commutative Banach algebra with the unit $e\in A$, then $u$ is constantly one on $\Sp{}{A}$. 
   The following lemma shows that the existence of a bounded approximate identity approximately plays a similar  role for a regular Tauberian algebra. 
  
\vskip2.0em
  
\begin{lemma}\label{l:1-bdd-approximate-identity}
Let $A$ be a regular commutative Tauberian Banach algebra. Then $A$ has a $D$-bounded approximate identity if and only if for each compact set $K\subseteq \Sp{}{A}$ and $\epsilon>0$, there is some $a_{K,\epsilon}\in A$ such that $\norm{a_{K,\epsilon}}_A\leq D$ and $\widehat{a}_{K,\epsilon}|_K \equiv 1$.
\end{lemma}

\begin{proof} Suppose that $(e_\alpha)_\alpha$ is a bounded approximate identity of $\cA$ such that $\norm{e_\alpha}_A \leq D$ for some $D>0$.
For each $K\subseteq \Sp{}{A}$, let $b_K\in A_c$ such that $b_K|_K\equiv 1$ and $I_K$ be the ideal $\{b\in A:\; \widehat{b}(K)=\{0\}\}$. Therefore, for each $b\in  A$, $bb_K -b_K \in I_K$. Considering the quotient norm of $\cA/I_K$, one gets  $\norm{b_K + I_K}_{A/I_K}=\lim_\alpha \norm{ b_Ke_\alpha + I_K}_{A/I_K} = \lim_\alpha \norm{ e_\alpha + I_K}_{A/I_K} \leq D$.
 So, there is some $b\in A_c \cap I_K$ such that $\norm{ b_K + b }_A < D+\epsilon$. Note that for $a_K:=(b_K+b)$, $a_K|_K\equiv 1$ and $a_K \in A_c$.
 
 Conversely, for each $\epsilon>0$ and $K\subseteq \Sp{}{A}$ compact, let ${a}_{K}\in A$ such that $\widehat{a}_{K}|_K \equiv 1$ and $\norm{a_{K}}_A \leq D (1 + \epsilon)$. Define $e_{K,\epsilon}:=(1+\epsilon)^{-1}a_{K}$. It is not hard to show that $(e_{K,\epsilon})_{K,\epsilon}$ forms an approximate identity of the Tauberian algebra $A$ which is $\norm{\cdot}_A$-bounded by $D$ where $\epsilon \rightarrow 0$ and $K \rightarrow \Sp{}{A}$.   
\end{proof}

Let $B$ be an abstract Segal algebra with respect to a Banach algebra $A$ and $A$ has a $\norm{\cdot}_A$-bounded approximate identity. Then, by an approximation argument, one can show that $A$ has a $\norm{\cdot}_A$-bounded approximate identity  which lies in $B$.
If $A$ is a Tauberian algebra, the $\norm{\cdot}_A$-bounded approximate identity of $A$ may belong to $A_c \cap B$.
The density  condition and relation of the norms implies that a proper abstract Segal algebra never has a bounded approximate identity.

For a Banach algebra $A$, let $ZA$ denote the \ma{center} of $A$ which is the commutative subalgebra of $A$ consisting of all elements $a\in A$ such that $ab=ba$ for every $b\in A$.
The following proposition proves the non-approximate amenability of Segal algebras with notable centres. %The proof is adapted the argument of the proof of \cite[Theorem~2.4]{ma}. 

\vskip2.0em

\begin{prop}\label{p:central-aa-of-abstract-Segal-algebras}
Let $B$ be a proper abstract Segal algebra of a Banach algebra $A$ which has a central bounded approximate identity and $ZB$ is dense in $ZA$. If $ZA$ is a regular Tauberian algebra, then $B$ is not approximately amenable.
\end{prop}

\begin{proof} 
To prove that such an abstract  Segal algebra is not approximately amenable, we apply the criterion developed in \cite{ch}. To do so, we should construct a sequence $(a_n)_{n\in \Bbb{N}}$ in $B$ such that is $\norm{\cdot}_A$-bounded, $\norm{\cdot}_B$-unbounded, and satisfying $a_na_{n+1}=a_{n+1}a_n=a_n$ for every $n\in \Bbb{N}$. 

Fist note that $ZB$ is an abstract Segal algebra of $ZA$.  For a fixed $\epsilon>0$ and $K_0 \subseteq \Delta({ZA})$,  for each compact set $K$ such that $K_0\subseteq K\subseteq \Delta(ZA)$, there is some $a_K\in ZA_c$ such that $\widehat{a}_K|_K\equiv 1$ and $\norm{a_K}_A \leq D+\epsilon$, by Lemma~\ref{l:1-bdd-approximate-identity}. Define the $\norm{\cdot}_A$-bounded net $(a_K)_{K_0\subseteq K \subseteq \Delta(ZA)}$ as above directed by inclusion over compact sets $K$. Therefore $a_{K_1} a_{K_2}= a_{K_1}$ if $\supp(\widehat{a}_{K_1})\subseteq K_2$.

We claim that $(a_K)_{K_0 \subseteq K \subseteq \Omega_{Z\cA}}$ is $\norm{\cdot}_B$-unbounded. Note that  $ZB$ is a Tauberian algebra. Therefore,  $A$ has a $\norm{\cdot}_A$-bounded approximate identity $(e_\alpha)\subseteq  ZA_c \cap ZB$.  So, for each $\alpha$, for $K=\supp(e_\alpha)$, $e_\alpha a_K =e_\alpha$. Hence,
\[
\norm{e_\alpha}_B = \lim_{K_0 \subseteq K \rightarrow \Delta(ZA)} \norm{e_\alpha a_K }_B \leq \limsup_{K_0 \subseteq K} \norm{e_\alpha}_A \norm{a_K}_B.
\]
Therefore, if $(a_K)_{K_0\subseteq K}$ is $\norm{\cdot}_B$-bounded, $(e_\alpha)_\alpha$ is a $\norm{\cdot}_B$-bounded approximate identity of $A$ which violates the properness of $B$.

 To generate a sequence which satisfies the desired conditions  mentioned before, fix a non-empty compact set $K_0\subseteq \Delta(ZA)$. By our claim, we inductively construct a sequence of compact sets $K_0 \subset K_1 \subset \cdots $ in $ \Delta(ZA)$ such that $a_{K_n} a_{K_{n-1}}=a_{K_{n-1}}$ and $\norm{B}{a_{K_n}} \geq n$ for all $n\in\Bbb{N}$. Then $B$ is not approximately amenable.  \end{proof}

Now we can prove the main theorem  of the paper.

\vskip1.0em

\noindent{\it Proof of Theorem~\ref{t:main-thm}.} 
Note that for every   SIN group $G$, $L^1(G)$ has a central bounded approximate identity. Moreover, for each Segal algebra $S^1(G)$, $ZS^1(G)$ is dense in $ZL^1(G)$, \cite[Theorem~2]{ko}.  On the other hand,    
\cite[Theorem~1.8]{mo} implies that $ZL^1(G)$ is a semisimple regular commutative Tauberian algebra. So Proposition~\ref{p:central-aa-of-abstract-Segal-algebras} can be applied to finish the proof. {\hfill$\Box$\newline\ignorespacesafterend}

\ques{Applying some results about structure of locally compact groups, Kotzmann, \cite{ko}, showed that for every Segal algebra $S^1(G)$, $ZS^1(G)$ is dense in $ZL^1(G)$. The group structure in his proof is essential. It seems that there is not any immediate argument to generalize this proof for a wider class  of abstract Segal algebras. It would be of interest if one can generalize this result to abstract Segal algebras. In other words, is there  an abstract Segal algebra whose centre is not dense in its ancestor?}

\ques{Is every proper Segal algebra of a locally compact group not approximately amenable?}
%%%%%%%%%%%%%%%%%%%%%%%%%%%%%%%%%%%%%%%%%%
\vskip1.0em

%%%%%%%%%%%%%%%%%%%%%%%%%%%%%%%%%%%%%%%%%%%%%%%%%%%%%%%%%%%%%%%%%%%%%%%%%%%%%%%%%%%%%%%%%%%%%%%%%%%%%%%%%

\noindent{\bf \textsf{Acknowledgements.}}
This research was supported by a Ph.D. Dean's Scholarship at  University of Saskatchewan and a Postdoctoral Fellowship from the Fields Institute  and University of Waterloo.  These supports are gratefully acknowledged. The author also would like to express his deep gratitude to Yemon Choi and Ebrahim Samei for many constructive discussions.
%%%%%%%%%%%%%%%%%%%%%%%%%%%%%%%%%%%%%%%%%%%%%%%%%%%%%%%%%%%%%%%%%%%%%%
\footnotesize

\bibliographystyle{plain}
\bibliography{Bibliography}

\def\cprime{$'$} \def\cprime{$'$}
\begin{thebibliography}{1}

\bibitem{ma}
Mahmood Alaghmandan.
\newblock Approximate amenability of {S}egal algebras.
\newblock {\em J. Aust. Math. Soc.}, 95(1):20--35, 2013.

\bibitem{ch}
Y.~Choi and F.~Ghahramani.
\newblock Approximate amenability of {S}chatten classes, {L}ipschitz algebras
  and second duals of {F}ourier algebras.
\newblock {\em Q. J. Math.}, 62(1):39--58, 2011.

\bibitem{da}
H.~G. Dales and R.~J. Loy.
\newblock Approximate amenability of semigroup algebras and {S}egal algebras.
\newblock {\em Dissertationes Math. (Rozprawy Mat.)}, 474:58, 2010.

\bibitem{aa}
F.~Ghahramani and R.~J. Loy.
\newblock Generalized notions of amenability.
\newblock {\em J. Funct. Anal.}, 208(1):229--260, 2004.

\bibitem{go2}
Fr{\'e}d{\'e}ric Gourdeau.
\newblock Amenability and the second dual of a {B}anach algebra.
\newblock {\em Studia Math.}, 125(1):75--81, 1997.

\bibitem{ka}
Eberhard Kaniuth.
\newblock {\em A course in commutative {B}anach algebras}, volume 246 of {\em
  Graduate Texts in Mathematics}.
\newblock Springer, New York, 2009.

\bibitem{ko}
Ernst Kotzmann and Harald Rindler.
\newblock Segal algebras on non-abelian groups.
\newblock {\em Trans. Amer. Math. Soc.}, 237:271--281, 1978.

\bibitem{mo}
J.~Liukkonen and R.~Mosak.
\newblock Harmonic analysis and centers of group algebras.
\newblock {\em Trans. Amer. Math. Soc.}, 195:147--163, 1974.

\bibitem{re2}
Hans Reiter.
\newblock {\em {$L^{1}$}-algebras and {S}egal algebras}.
\newblock Lecture Notes in Mathematics, Vol. 231. Springer-Verlag, Berlin,
  1971.

\end{thebibliography}
%%%%%%%%%
%%%%%%%%%%%%%%%%%%%%%%%%%%%%%%%%%%%%%%%%%%%%%%%%%%%%%%%%%%%%%%%%%%%%%%%%%%%%%%%
\vskip1.5em

{ 
{\normalsize Mahmood Alaghmandan}\newline\indent
\vskip0.1em

\footnotesize   The Fields Institute For Research In Mathematical Sciences,\newline\indent
222 College St,\newline\indent
Toronto, ON M5T 3J1, Canada}

\vskip1em
\&

\vskip1em

Department of Pure Mathematics, \newline\indent
University of Waterloo, \newline\indent
Waterloo, ON N2L
3G1, Canada

\vskip1em

Email: \texttt{m.alaghmandan@utoronto.ca}

\end{document}